\documentclass[11pt]{amsart}

\usepackage[a4paper,margin=30mm]{geometry}
\usepackage[T1]{fontenc}
\usepackage{lmodern}
\usepackage{microtype}
\usepackage{amsmath,amssymb,amsthm,mathtools, MnSymbol}
\usepackage{aliascnt}
\usepackage{enumitem}
\usepackage{xcolor}
\usepackage{hyperref}
\usepackage[nameinlink,noabbrev]{cleveref}

\hypersetup{
  colorlinks=true,
  linkcolor=blue!45!black,
  citecolor=blue!45!black,
  urlcolor=blue!55!black,
  pdftitle={On free generators in the Grothendieck--Teichmüller Lie Algebra},
  pdfauthor={Thomas Willwacher}
}

\setlist{itemsep=0.25em,topsep=0.5em}
\allowdisplaybreaks

\theoremstyle{plain}
\newtheorem{theorem}{Theorem}[section]
\newaliascnt{proposition}{theorem}
\newtheorem{proposition}[proposition]{Proposition}
\aliascntresetthe{proposition}
\newaliascnt{lemma}{theorem}
\newtheorem{lemma}[lemma]{Lemma}
\aliascntresetthe{lemma}
\newaliascnt{corollary}{theorem}

\aliascntresetthe{corollary}

\theoremstyle{definition}
\newaliascnt{definition}{theorem}
\newtheorem{definition}[definition]{Definition}
\aliascntresetthe{definition}
\newaliascnt{convention}{theorem}

\aliascntresetthe{convention}

\theoremstyle{remark}
\newaliascnt{remark}{theorem}
\newtheorem{remark}[remark]{Remark}
\aliascntresetthe{remark}

\newaliascnt{example}{theorem}
\newtheorem{example}[example]{Example}
\aliascntresetthe{example}

\crefname{theorem}{Theorem}{Theorems}
\crefname{proposition}{Proposition}{Propositions}
\crefname{lemma}{Lemma}{Lemmas}
\crefname{corollary}{Corollary}{Corollaries}
\crefname{definition}{Definition}{Definitions}
\crefname{remark}{Remark}{Remarks}
\crefname{convention}{Convention}{Conventions}
\crefname{example}{Example}{Examples}
\crefname{section}{Section}{Sections}

\newcommand{\Q}{\mathbb Q}
\newcommand{\C}{\mathbb C}
\newcommand{\Z}{\mathbb Z}
\newcommand{\GRT}{\mathrm{GRT}}
\newcommand{\grt}{\mathfrak{grt}}

\newcommand{\Lhat}{\widehat{\mathbb L}}
\newcommand{\fhat}{\widehat{\mathfrak f}}
\newcommand{\Ass}{\operatorname{Ass}_1}
\newcommand{\wt}{\operatorname{wt}}
\newcommand{\lev}{\operatorname{lev}}
\newcommand{\Ihara}{\mathbin{\circledast}}
\newcommand{\shuffle}{\mathbin{\amalg}}
\newcommand{\two}[1]{2^{\{#1\}}}
\newcommand{\K}{\mathbb{K}}
\newcommand{\UZ}{\mathcal U^{\mathcal Z}}

\title{On free generators in the\\
Grothendieck--Teichm\"uller Lie Algebra}

\date{}

\author{Thomas Willwacher}
\address{Department of Mathematics\\ ETH Zurich\\ Rämistrasse 101 \\ 8092 Zurich, Switzerland}
\email{thomas.willwacher@math.ethz.ch}

\begin{document}

\begin{abstract}
We prove that any homogeneous family
$\sigma_3,\sigma_5,\ldots$ in $\grt_1(\K)$ with nonzero
coefficients of $x^{2k}y$ in $\sigma_{2k+1}$ generates a free Lie subalgebra, building on the work of Brown.
\end{abstract}

\maketitle


\section{Introduction}\label{sec:intro}

The Grothendieck--Teichm\"uller Lie algebra $\grt_1(\K)$ is a weight-graded pronilpotent Lie algebra that was introduced by Drinfeld in \cite{Drinfeld}.
Concretely, elements of $\grt_1(\K)$ are Lie-like formal power series $\psi\in\K\llangle x,y\rrangle$ in two non-commuting variables $x$ and $y$ with coefficients in the field $\K$ that satisfy the linear equations \eqref{eq:grt-sym}-\eqref{eq:grt-pentagon} below.

The Deligne-Drinfeld conjecture states that $\grt_1$ is the completed free Lie algebra in generators $\sigma_3,\sigma_5,\sigma_7,\dots$ of homogeneous weights (degree in $x,y$) $3,5,7,\dots$.
It is known from work of Drinfeld that there exist elements $\sigma_{2k+1}\in\grt_1(\K)$ of weight $2k+1$ for all $k\ge 1$ such that the coefficient $[x^{2k}y]\sigma_{2k+1}$ of $x^{2k}y$ in $\sigma_{2k+1}$ is non-zero, and those are expected to be the generators of $\grt_1$ according to the Deligne-Drinfeld conjecture.
The elements $\sigma_{2k+1}$ are not uniquely determined beyond their leading coefficients. 


The most important result in the study of the Grothendieck-Teichmüller Lie algebra (in the opinion of the author) is the proof of one half of the Deligne-Drinfeld conjecture by Francis Brown:
\begin{theorem}[Consequence of F. Brown's work \cite{BrownMTZ}]
  \label{thm:brown}
    Let $\K$ be a field of characteristic zero. 
    Then there are elements $\sigma_3,\sigma_5,\dots\in \grt_1(\K)$ with non-zero leading coefficients as above that generate a free Lie subalgebra of $\grt_1(\K)$.
\end{theorem}

The theorem as stated above has two downsides: First, it does not readily appear in the above form in \cite{BrownMTZ}, though it is a standard consequence of the main results of that paper.
Second, we do not have very good control over the exact form of the generators $\sigma_{2k+1}\in \grt_1(\K)$ that are required to make it true.
This paper tries to remedy both of these issues by giving a somewhat self-contained and motive-free proof of the following slightly improved version of Brown's theorem:
\begin{theorem}\label{thm:main}
    Let $\K$ be a field of characteristic zero. 
    Let $\sigma_3,\sigma_5,\dots$ be any collection of elements of the Grothendieck--Teichm\"uller Lie algebra $\grt_1(\K)$ such that $\sigma_{2k+1}$ is homogeneous of weight $2k+1$ and the leading order coefficients $[x^{2k}y]\sigma_{2k+1}$ are non-zero for all $k\ge 1$. Then the elements $\sigma_3,\sigma_5,\dots$ generate a free Lie subalgebra of $\grt_1(\K)$.
\end{theorem}

We emphasize that the main technical ideas of the proof are due to Francis Brown \cite{BrownMTZ} and we claim no originality in that regard. We also use Terasoma's theorem that the Brown--Zagier coefficient identity holds for every Drinfeld associator \cite{Terasoma,Zagier}.
The main contribution of this paper is to make Brown's proof more accessible to non-motivic readers and to state it in a way that works for any choice of generators $\sigma_{2k+1}$ with non-zero leading coefficients.
This work is loosely based on lecture notes of the author for a course given at ETH Zurich in 2012 on the Grothendieck-Teichmüller group and Brown's Theorem \cite{Willwacher}.

\begin{remark}
Although Theorem \ref{thm:main} is arguably only a minor improvement of Brown's Theorem \ref{thm:brown}, it can be useful in applications where it is difficult to control the behavior of the generators $\sigma_{2k+1}$ beyond their leading coefficients.
To give one prominent example, Brown--Chan--Galatius--Payne construct lower bounds on the cohomology of the general linear groups and the moduli spaces of abelian varieties in \cite{BCGP} using the existence of free generators in $\grt_1$. Since they can control only the leading coefficients, see \cite[Proposition 6.4]{BCGP}, their results \cite[Theorem 1.3 and Proposition 1.8]{BCGP} are restricted to a range of weights ("$N=11$") that reflects the range in which the Deligne-Drinfeld conjecture is known to hold. Our Theorem \ref{thm:main} allows us to remove this restriction and thus to extend their results to $N=\infty$. 
\end{remark}

\subsection*{Declaration of AI usage}
This document has been created with the assistance of AI tools, namely ChatGPT 5.6 Pro and to a lesser extent GitHub Copilot. Concretely, the workflow used was the following:
In a first stage of brainstorming the project was dicussed with the AI, on the basis of making my old lecture notes on Brown's work into a paper, with the upgrade of replacing the role of the KZ associator in those notes by a more general one.
After the exact plan was decided, I used ChatGPT 5.6 Pro to generate a first draft of the paper on the basis of the lecture notes, which was then iteratively improved in several rounds of prompting. Once this resulted in a somewhat acceptable text, I took over the editing and revised and reworked the manuscript. In the final stage ChatGPT was used for proofreading and for critiquing my writing. The human author is fully responsible for any errors or shortcomings that might remain in the final version of the manuscript.

\section{Recollections on the Grothendieck--Teichm\"uller Lie algebra}\label{sec:background}

In the following $\K$ will always denote a field of characteristic zero.

\subsection{Power series}
We denote by $\K\llangle x,y\rrangle$
the associative algebra of formal power series in two non-commuting variables $x$ and $y$.
It comes with a complete grading by the \emph{weight}, which is the total number of letters $x$ and $y$ in a monomial.
With this grading defining a topology, $\K\llangle x,y\rrangle$ is a complete Hopf algebra.
Its counit is
$\varepsilon:\K\llangle x,y\rrangle\to \K$. A formal series $H$ is
\emph{unit-normalized} if $\varepsilon(H)=1$. The standard completed Hopf coproduct is
\[
  \Delta_{\mathrm{st}}(x)=x\otimes1+1\otimes x,
  \qquad
  \Delta_{\mathrm{st}}(y)=y\otimes1+1\otimes y.
\]
The antipode $S$ is, for $H\in \K\llangle x,y\rrangle$,
\[
  S(H)(x,y)=Rev(H(-x,-y)),
\]
where $Rev(-)$ acts by reversing the order of letters in a monomial.

For a monomial $w$ in $x$ and $y$, and a series
$H\in\K\llangle x,y\rrangle$, we write $[w]H\in\K$ for the
coefficient of $w$ in $H$.

\subsection{Ihara product and Ihara bracket}
An element $g\in\K\llangle x,y\rrangle$ is \emph{group-like} if
\begin{align*}
  \Delta_{\mathrm{st}}(g)&=g\otimes g,
  &
  \varepsilon(g)&=1.
\end{align*}
The group-like elements form a group under the standard concatenation product. For $g\in \K\llangle x,y\rrangle$ group-like and $H\in \K \llangle x,y\rrangle$ a
unit-normalized series, define
\begin{equation}\label{eq:ambient-ihara-action}
  (g\Ihara H)(x,y)
  =g(x,y)H\bigl(x,g^{-1}_{\mathrm{conc}}(x,y)yg(x,y)\bigr),
\end{equation}
where $g^{-1}_{\mathrm{conc}}$ denotes the inverse of $g$
for the ordinary concatenation product.
If $H$ is itself group-like then $g\Ihara H$ is also group-like.  The Ihara product $\Ihara$ is associative and has unit $1$ and endows the set of group-like series with a(nother) group structure.

An element $\psi\in \K\llangle x,y\rrangle$ is \emph{Lie-like} if
\[
  \Delta_{\mathrm{st}}(\psi)=\psi\otimes1+1\otimes\psi.
\]
The Lie-like elements form a Lie algebra under the standard (concatenation) commutator bracket $[-,-]$.
For $\psi$ Lie-like define a derivation $d_\psi$ of $\K\llangle x,y\rrangle$ by
\[
  d_\psi(x)=0,\qquad d_\psi(y)=[y,\psi].
\]
The Ihara bracket of Lie-like elements $\psi,\phi\in \K\llangle x,y\rrangle$ is then
\begin{equation}\label{eq:ihara-bracket}
  \{\psi,\phi\}=[\psi,\phi]+d_\psi(\phi)-d_\phi(\psi).
\end{equation}

Both the standard concatenation bracket and the Ihara bracket preserve the weight grading.

\subsection{The Lie algebra \texorpdfstring{$\mathfrak{grt}_1$}{grt1}}
For $n\geq2$, let $\mathfrak t_n$ be the degree-completed Lie algebra on
symbols $t_{ij}=t_{ji}$, $i\neq j$, modulo the relations
\[
 [t_{ij},t_{kl}]=0\quad(i,j,k,l\text{ distinct}),
 \qquad
 [t_{ij},t_{ik}+t_{jk}]=0\quad(i,j,k\text{ distinct}).
\]
Then the \emph{Grothendieck--Teichmüller Lie algebra} $\grt_1(\K)\subset \K\llangle x,y\rrangle$ is the vector space of Lie-like elements $\psi$ that satisfy the equations
\begin{align}
  \psi(x,y)+\psi(y,x)&=0, \label{eq:grt-sym}\\
  \psi(x,y)+\psi(y,-x-y)+\psi(-x-y,x)&=0 ,\label{eq:grt-cycle}\\
  \psi(t_{12},t_{23})-\psi(t_{12},t_{23}+t_{24})
  +\psi(t_{12}+t_{13},t_{24}+t_{34}) & &  \notag\\
  -\psi(t_{13}+t_{23},t_{34})+\psi(t_{23},t_{34}) &= 0 \quad \text{in $\mathfrak t_4$}&
  \label{eq:grt-pentagon}.
\end{align}
We equip $\grt_1(\K)$ with the Ihara bracket \eqref{eq:ihara-bracket}. This makes $\grt_1(\K)$ into a pronilpotent Lie algebra, which is graded by the weight. The weight-$n$ part of $\grt_1(\K)$ is denoted by $\grt_1(\K)_n$. In each degree the defining equations are rational linear equations in a
finite-dimensional vector space.  Therefore
\begin{equation}\label{eq:base-change}
  \grt_1(\K)_n=\grt_1(\Q)_n\otimes_{\Q}\K
\end{equation}
for every characteristic-zero field $\K$.
We shall also use the standard fact that normalized elements exist in every
odd weight.

\begin{lemma}\label{lem:ell-nonzero}
For every \(k\geq1\), there exists a homogeneous element
\(
  \sigma_{2k+1}\in\grt_1(\Q)_{2k+1}
\)
such that
\(
  [x^{2k}y]\sigma_{2k+1}\neq0
\).
\end{lemma}

\begin{proof}[Reference]
This is due to Drinfeld \cite{Drinfeld}.  A construction from the
Knizhnik--Zamolodchikov associator is described in
\cite[Chapter~5]{Willwacher}, for $\C$ replacing $\Q$. But the defining equations of \(\grt_1\) are
rational in each weight, so the existenvce of an element $\sigma_{2k+1}\in \grt_1(\C)$, together with
\eqref{eq:base-change}, implies the existence of an element over \(\Q\) as well.
\end{proof}

If $\mathfrak h$ is a degree-complete pronilpotent Lie algebra,
let $\operatorname{Exp}(\mathfrak h)$ denote the corresponding prounipotent
group. The \emph{Grothendieck-Teichmüller group} is the prounipotent group
\[
  \GRT_1(\K)=\operatorname{Exp}(\grt_1(\K))\subset \K\llangle x,y\rrangle.
\]
Its group law is the Ihara product and its Lie algebra is $\grt_1(\K)$.

\subsection{Drinfeld associators}
A group-like element $\Phi\in\K\llangle x,y\rrangle$ is a \emph{Drinfeld associator} if it satisfies the equations
\begin{align}
  \Phi(x,y)\Phi(y,x)&=1,\label{eq:ass-sym}\\
  e^{z/2}\Phi(x,y)e^{x/2}\Phi(y,z)e^{y/2}\Phi(z,x)&=1
       \qquad(\text{with $z:=-x-y$}),\label{eq:ass-hex}\\
  \Phi(t_{12},t_{23}+t_{24})\Phi(t_{13}+t_{23},t_{34})
  &=\Phi(t_{23},t_{34})\Phi(t_{12}+t_{13},t_{24}+t_{34})
    \Phi(t_{12},t_{23}).\label{eq:ass-pent}
\end{align}
We write $\Ass(\K)\subset\K\llangle x,y\rrangle$ for the set of Drinfeld associators.
The prounipotent group $\GRT_1(\K)$ acts freely and transitively on
$\Ass(\K)$ \cite{Drinfeld}; its action is the restriction of
\eqref{eq:ambient-ihara-action}.

A Drinfeld associator $\Phi$ is \emph{even} if $\Phi(x,y)=\Phi(-x,-y)$. Equivalently, this means that the coefficients in $\Phi$ of all monomials in $x,y$ of odd length vanish.

\begin{lemma}\label{lem:even-associator}
There exists a rational even Drinfeld associator, i.e.,
$\Phi_{\mathrm{ev}}\in\Ass(\Q)$ such that
$\Phi_{\mathrm{ev}}(-x,-y)=\Phi_{\mathrm{ev}}(x,y)$.
\end{lemma}

\begin{proof}
This is part of Drinfeld's existence theorem for associators
\cite{Drinfeld}, see also \cite{APS}.  We fix one such associator for the rest of the paper.
\end{proof}


Below we will use the following elementary and well-known property of associators and elements of $\GRT_1(\K)$.

\begin{lemma}\label{lem:no-pure-words}
For every $\Phi\in\Ass(\K)$ and every $g\in\GRT_1(\K)$,
\[
  \Phi(x,0)=\Phi(0,y)=g(x,0)=g(0,y)=1.
\]
In other words, the coefficients of the words of the form $x^n$ or $y^n$ (with $n\geq 1$) in $\Phi$ or $g$
vanish.
\end{lemma}

\begin{proof}
A group-like series in one variable is of the form $\exp(cx)$.
It therefore suffices to prove that the linear coefficients vanish.

Let $\Phi_1=ax+by$ be the degree-one part of an associator.
The symmetry equation gives $b=-a$, and the degree-one part of
the pentagon equation gives
\[
  a(t_{12}-t_{34})=0
  \qquad\text{in }\mathfrak t_4.
\]
Hence $a=b=0$.

The same calculation applied to a degree-one element of
$\mathfrak{grt}_1(\K)$ shows that
$\mathfrak{grt}_1(\K)_1=0$. Since
$\GRT_1(\K)=\operatorname{Exp}(\mathfrak{grt}_1(\K))$,
every element of $\GRT_1(\K)$ has vanishing linear part.
The one-variable restrictions are therefore equal to $1$.
\end{proof}

\section{The free group and coefficient functions}\label{sec:coordinates}

\subsection{The free (source) group and its coordinate algebra}

Let $s_3,s_5,s_7,\ldots$ be symbols and denote the complete weight-graded free Lie algebra and its exponential group by 
\[
  \fhat=\Lhat_{\K}(s_3,s_5,s_7,\ldots),
  \qquad F=\operatorname{Exp}(\fhat).
\]
Here $s_{m}$ is defined to have weight $m$.

\begin{definition}\label{def:normalized-family}
A normalized family over $\K$ is a collection
\[
  (\sigma_m)_{m\geq3,\ m\text{ odd}},
  \qquad \sigma_m\in\grt_1(\K)_m,
  \qquad [x^{m-1}y]\sigma_m=1.
\]
\end{definition}

For a fixed such family, the universal property of $\fhat$ gives a continuous graded Lie algebra homomorphism
\begin{equation}\label{eq:source-lie-map}
  \rho_{\mathrm{Lie}}:\fhat\longrightarrow\grt_1(\K),
  \qquad s_m\longmapsto\sigma_m,
\end{equation}
and exponentiation gives a homomorphism of prounipotent groups
\begin{equation}\label{eq:source-group-map}
  \rho_F:F\longrightarrow\GRT_1(\K).
\end{equation}
We will call $F$ the \emph{source group} because it is the source of the map $\rho_F$.
The coordinate Hopf algebra of $F$ is the shuffle algebra
\begin{equation}\label{eq:coordinate-ring}
  \mathcal U_{\K}=\mathcal O(F)
  =\K\langle f_3,f_5,f_7,\ldots\rangle_{\shuffle},
\end{equation}
that we will call the \emph{source algebra}.
Here $f_m$ is dual to the generator $s_m$ and normalized by $f_m(\exp(t s_n))=t \delta_{m,n}$.
As a vector space $\mathcal U_{\K}$ has a basis consisting of the words $f_{i_1}\cdots f_{i_r}$.
Here juxtaposition denotes a basis word, while multiplication is the shuffle product.
The coproduct and counit are
\begin{align}
  \Delta(f_{i_1}\cdots f_{i_r})
  &=\sum_{j=0}^{r}f_{i_1}\cdots f_{i_j}\otimes
                    f_{i_{j+1}}\cdots f_{i_r},\label{eq:deconcat-coproduct}\\
  \varepsilon(f_{i_1}\cdots f_{i_r})
  &=\begin{cases}1,&r=0,\\0,&r>0.\end{cases}
\end{align}
Thus $(\Delta A)(h_1,h_2)=A(h_1h_2)$ for any $A\in \mathcal U_{\K}$, $h_1,h_2\in F$.  We write $1_F$ for the identity
of $F$. Evaluation at the identity is the counit: $A(1_F)=\varepsilon(A)$.

The weight of a word $f_{i_1}\cdots f_{i_r}\in \mathcal U_{\K}$ is $i_1+\cdots+i_r$.  Let
$\mathcal U_{\K,N}$ be the homogeneous part of weight $N$, and put
$W_{\leq N}\mathcal U_\K=\bigoplus_{j\leq N}\mathcal U_{\K,j}$.
For every odd $m\geq3$, define
\begin{equation}\label{eq:lambda}
  \lambda_m(A)=[f_m]A\in \K,
  \qquad
  \partial_m=(\lambda_m\otimes\operatorname{id})\Delta \colon \mathcal U_{\K}\to\mathcal U_{\K}.
\end{equation}
Equivalently,
\[
  \partial_m(f_{i_1}\cdots f_{i_r})=
  \begin{cases}
    f_{i_2}\cdots f_{i_r},&i_1=m,\\
    0,&i_1\neq m.
  \end{cases}
\]
For a word $u=m_1\cdots m_r$ in the alphabet $\{3,5,7,\dots\}$ set
\begin{equation}\label{eq:partial-word}
  \partial_u=\partial_{m_1}\circ\cdots\circ\partial_{m_r}.
\end{equation}

\subsection{Scalar coefficients and orbit functions}

For a word $w$ in $x,y$, the symbol $[w]H\in \K$ denotes its coefficient in
$H\in\K\llangle x,y\rrangle$.  Put $x_0=x$ and $x_1=y$.  For $N>0$ and
$a_0,\ldots,a_{N+1}\in\{0,1\}$, define the numbers
\begin{align}
  I_H(0;a_1,\ldots,a_N;1)
  &:=[x_{a_N}\cdots x_{a_1}]H,\label{eq:I-01}\\
  I_H(1;a_1,\ldots,a_N;0)
  &:=(-1)^N I_H(0;a_N,\ldots,a_1;1),\label{eq:I-reflection}\\
  I_H(a;a_1,\ldots,a_N;a)&:=0.\label{eq:I-equal-endpoints}
\end{align}
For the word with empty interior put
\begin{equation}\label{eq:I-empty}
  I_H(a_0;\,;a_1):=\varepsilon(H).
\end{equation}
The endpoints are not letters of the coefficient word.  In particular,
\begin{equation}\label{eq:leading-I}
  I_H(0;1,0^{N-1};1)=[x^{N-1}y]H.
\end{equation}

Here $0^q$ means a string of $q$ zeros, with $0^0$ the empty string.
For a finite nonempty sequence $\mathbf n=(n_1,\ldots,n_r)$ of positive
integers, put
\[
  \beta(\mathbf n)
  =1,0^{n_1-1},\ldots,1,0^{n_r-1}
\]
and define
\begin{equation}\label{eq:Z-Psi-def}
  Z_H(n_1,\ldots,n_r):=I_H(0;\beta(\mathbf n);1),
  \qquad Z_H(\varnothing):=\varepsilon(H).
\end{equation}
All quantities defined so far in this definition are scalars in $\K$.

Next let $\Phi_0\in\K\llangle x,y\rrangle$ be unit-normalized. Define functions on the group $F$ of the preceding subsection
by
\begin{align}
  \mathcal I_{\Phi_0,\rho}(a_0;a_1,\ldots,a_N;a_{N+1})(h)
  &:=I_{\rho_F(h)\Ihara\Phi_0}
       (a_0;a_1,\ldots,a_N;a_{N+1}),\label{eq:torsor-I}\\
  \mathcal I^F_\rho(a_0;a_1,\ldots,a_N;a_{N+1})(h)
  &:=I_{\rho_F(h)}(a_0;a_1,\ldots,a_N;a_{N+1})\label{eq:group-I}
\end{align}
for $h\in F$. Note that here we use the map $\rho_F$ from \eqref{eq:source-group-map}, and this in particular depends on the choice of the normalized family $(\sigma_m)_{m\geq3,\ m\text{ odd}}$, or equivalently the map $\rho_{\mathrm{Lie}}$ of \eqref{eq:source-lie-map}. We furthermore use here the convention that we denote the function in $\mathcal U_\K$ by the calligraphic letter $\mathcal I$, and the corresponding coefficients (numbers) by the regular letter $I$.

Because $\rho_{\mathrm{Lie}}$ is graded, a coefficient with $N$ interior
symbols satisfies
\begin{equation}\label{eq:weight-bound-general}
  \mathcal I^F_\rho(a_0;a_1,\ldots,a_N;a_{N+1})\in\mathcal U_{\K,N},
  \qquad
  \mathcal I_{\Phi_0,\rho}(a_0;a_1,\ldots,a_N;a_{N+1})
  \in W_{\leq N}\mathcal U_\K.
\end{equation}
For odd $m$,
\begin{equation}\label{eq:lambda-infinitesimal}
  \lambda_m \mathcal I^F_\rho(a_0;a_1,\ldots,a_m;a_{m+1})
  =I_{\sigma_m}(a_0;a_1,\ldots,a_m;a_{m+1}) \in \K.
\end{equation}

\subsection{The general cut formula}

\begin{proposition}[Cut formula for the coproduct]
\label{prop:cut-formula}
Let $\Phi_0\in\K\llangle x,y\rrangle$ be any unit-normalized formal series.  Fix
$N\geq0$ and $a_1,\ldots,a_N\in\{0,1\}$, and put
$a_0=0$, $a_{N+1}=1$.  Then
\begin{align}
&\Delta \mathcal I_{\Phi_0,\rho}(0;a_1,\ldots,a_N;1)\notag\\
&=\sum_{r=0}^{N}
  \sum_{0=i_0<i_1<\cdots<i_r<i_{r+1}=N+1}
  \left(\prod_{q=0}^{r}
  \mathcal I^F_\rho(a_{i_q};a_{i_q+1},\ldots,
                 a_{i_{q+1}-1};a_{i_{q+1}})\right)\notag\\
&\hspace{55mm}\otimes
  \mathcal I_{\Phi_0,\rho}(0;a_{i_1},\ldots,a_{i_r};1) \in \mathcal U_\K \otimes  \mathcal U_\K .
\label{eq:coaction}
\end{align}
Consequently, for every odd $m\geq3$,
\begin{align}
&\partial_m \mathcal I_{\Phi_0,\rho}(0;a_1,\ldots,a_N;1)\notag\\
&=\sum_{q=0}^{N-m}
  I_{\sigma_m}(a_q;a_{q+1},\ldots,a_{q+m};a_{q+m+1})\notag\\
&\hspace{28mm}\cdot
  \mathcal I_{\Phi_0,\rho}
  (0;a_1,\ldots,a_q,a_{q+m+1},\ldots,a_N;1) \in \mathcal U_\K,
\label{eq:contiguous-cut}
\end{align}
where the sum is empty if $N<m$.
\end{proposition}

\begin{proof}
We first prove a coefficient identity in the ambient series algebra.  Let
$g\in\K\llangle x,y\rrangle$ be group-like and let $H\in\K\llangle x,y\rrangle$ be unit-normalized.  Expanding
\eqref{eq:ambient-ihara-action} gives
\begin{align}
&I_{g\Ihara H}(0;a_1,\ldots,a_N;1)\notag\\
&=\sum_{r=0}^{N}
  \sum_{0=i_0<i_1<\cdots<i_r<i_{r+1}=N+1}
  \left(\prod_{q=0}^{r}
  I_g(a_{i_q};a_{i_q+1},\ldots,
              a_{i_{q+1}-1};a_{i_{q+1}})\right)
  I_H(0;a_{i_1},\ldots,a_{i_r};1).
\label{eq:universal-coefficient-identity}
\end{align}
To see this directly, first reduce by linearity to a single word of $H$.
In that word, replace every $x$ by $x$ and every $y$ by
$g^{-1}_{\mathrm{conc}}yg$, and multiply by the initial factor $g$. In the word(s) thus produced, choose
the positions
$i_1<\cdots<i_r$ of the distinguished $x$ or $y$ letters supplied by the
chosen word of $H$. (I.e., those letters that were not created by the insertions of copies of $g$ or $g^{-1}_{\mathrm{conc}}$.)  The letters between two consecutive distinguished
letters form one contiguous block.  After adjacent factors $g$ and
$g^{-1}_{\mathrm{conc}}$ cancel, such a block is supplied by $g$, by
$g^{-1}_{\mathrm{conc}}$, or by the identity, according to its two endpoint
symbols.  Since $g$ is group-like, $g^{-1}_{\mathrm{conc}}=S(g)$ for the
standard antipode, and
\[
  S(x_{b_1}\cdots x_{b_s})=(-1)^s x_{b_s}\cdots x_{b_1}.
\]
Consequently the block coefficient is exactly the corresponding quantity
$I_g(\cdots)$; a nonempty block with equal endpoints
comes from the identity and is zero.  Multiplying all block coefficients
and summing over the retained positions proves
\eqref{eq:universal-coefficient-identity}.

For $h_1,h_2\in F$, associativity of the action and the homomorphism property
of $\rho_F$ give
\[
  \rho_F(h_1h_2)\Ihara\Phi_0
  =\rho_F(h_1)\Ihara(\rho_F(h_2)\Ihara\Phi_0).
\]
Apply \eqref{eq:universal-coefficient-identity} with
$g=\rho_F(h_1)$ and $H=\rho_F(h_2)\Ihara\Phi_0$.  By the definition of the coproduct on $\mathcal U_\K$, this is \eqref{eq:coaction}.

Finally apply $(\lambda_m\otimes\operatorname{id})$ to
\eqref{eq:coaction}.  A one-letter word in a shuffle product can come from
exactly one nonconstant factor.  By \eqref{eq:weight-bound-general}, that
factor must correspond to one complementary block with exactly $m$
interior symbols.  Equation \eqref{eq:lambda-infinitesimal} then gives
\eqref{eq:contiguous-cut}.
\end{proof}

\begin{example}[The case of two interior symbols]
For $N=2$, formula \eqref{eq:coaction} reads
\begin{align*}
\Delta \mathcal I_{\Phi_0,\rho}(0;a_1,a_2;1)
={}&\mathcal I^F_\rho(0;a_1,a_2;1)\otimes1\\
&+\mathcal I^F_\rho(a_1;a_2;1)\otimes \mathcal I_{\Phi_0,\rho}(0;a_1;1)\\
&+\mathcal I^F_\rho(0;a_1;a_2)\otimes \mathcal I_{\Phi_0,\rho}(0;a_2;1)\\
&+1\otimes \mathcal I_{\Phi_0,\rho}(0;a_1,a_2;1).
\end{align*}
The four terms correspond respectively to retaining from the second factor
none, only $a_1$, only $a_2$, or both interior symbols.
\end{example}

\subsection{Elementary source-algebra facts}

\begin{lemma}[Kernel lemma]\label{lem:kernel}
Let $A\in W_{\leq N}\mathcal U_\K$.  If $\varepsilon(A)=0$ and
$\partial_mA=0$ for every odd $m<N$, then
\[
  A=\begin{cases}
      c f_N,&N\text{ odd},\\
      0,&N\text{ even},
    \end{cases}
\]
for some $c\in \K$.
\end{lemma}

\begin{proof}
Expand $A$ in the word basis.  Every nonempty word of weight at most $N$
starts with some $f_m$ with $m<N$, except for the one-letter word $f_N$.
The coefficient of every tail following $f_m$ is detected by $\partial_m$,
and the empty-word coefficient is $\varepsilon(A)$.
\end{proof}

For $\xi\in\fhat$, define
\begin{equation}\label{eq:Dxi}
  (D_\xi A)(h)=\left.\frac{d}{dt}\right|_{t=0}
  A(\exp(t\xi)h).
\end{equation}
Then $D_{s_m}=\partial_m$.

\begin{lemma}[Anti-representation sign]\label{lem:anti}
For $\xi,\eta\in\fhat$,
\[
  [D_\xi,D_\eta]=-D_{[\xi,\eta]}.
\]
If $L$ is a Lie monomial of bracket length $r$, and $\partial_L$ is formed
by replacing letters by the corresponding $\partial_m$ and brackets by
operator commutators, then
\begin{equation}\label{eq:anti-sign}
  \partial_L=(-1)^{r-1}D_L.
\end{equation}
\end{lemma}

\begin{proof}
The vector fields in \eqref{eq:Dxi} are right invariant, whose bracket is
the negative of the Lie bracket.  The second assertion follows by induction.
\end{proof}

\subsection{Specialization to the even associator}

Fix $\Phi_0=\Phi_{\mathrm{ev}}$ to be some even associator as in Lemma \ref{lem:even-associator}.  For brevity, we write
\begin{equation}\label{eq:I-abbreviations}
  \mathcal I:=\mathcal I_{\Phi_{\mathrm{ev}},\rho},
  \qquad \mathcal I^F:=\mathcal I^F_\rho.
\end{equation}
Every point $\rho_F(h)\Ihara\Phi_{\mathrm{ev}}$ belongs to $\Ass(\K)$ and
is therefore group-like, so the functions $\mathcal I$ satisfy the shuffle identities.
Likewise every $\rho_F(h)\in\GRT_1(\K)$ is group-like, so the functions $\mathcal I^F$
satisfy shuffle identities.  By Lemma \ref{lem:no-pure-words} one has
\begin{equation}\label{eq:I-constant-word}
 \mathcal I(a_0;a_1,\ldots,a_N;a_{N+1})
 =\mathcal I^F(a_0;a_1,\ldots,a_N;a_{N+1})=0
\end{equation}
when $N>0$ and $a_1=\cdots=a_N$.  Evenness gives
\begin{equation}\label{eq:odd-constant-term}
  \varepsilon\bigl(\mathcal I(0;a_1,\ldots,a_N;1)\bigr)=0
  \qquad(N\text{ odd}).
\end{equation}
Also
\begin{equation}\label{eq:weight-bound}
 \mathcal I(a_0;a_1,\ldots,a_N;a_{N+1})\in W_{\leq N}\mathcal U_\K.
\end{equation}

For a finite sequence of positive integers, define the function
\begin{equation}\label{eq:Z-def}
 \mathcal Z(n_1,\ldots,n_r)
 :=\mathcal I(0;1,0^{n_1-1},\ldots,1,0^{n_r-1};1)
 \in\mathcal U_\K,
 \qquad \mathcal Z(\varnothing)=1.
\end{equation}
For every $h\in F$,
\begin{equation}\label{eq:Z-scalar-function-relation}
  \mathcal Z(n_1,\ldots,n_r)(h)
  =Z_{\rho_F(h)\Ihara\Phi_{\mathrm{ev}}}(n_1,\ldots,n_r).
\end{equation}
Thus $Z_\Phi(n_1,\ldots,n_r)$ is a scalar attached to one fixed series
$\Phi$, whereas $\mathcal Z(n_1,\ldots,n_r)$ is a regular function on $F$.
For a word $w$ in $\{2,3\}$, define
\[
  \varrho(2)=10,\qquad \varrho(3)=100,
\]
and extend by concatenation.  Then
\begin{equation}\label{eq:Z-word}
  \mathcal Z(w)=\mathcal I(0;\varrho(w);1).
\end{equation}
For $q\geq0$, $\two{q}$ denotes the sequence of $q$ copies of $2$.

\begin{lemma}[Single odd coefficient and even strings]\label{lem:single-zeta}
For every $n\geq1$,
\begin{equation}\label{eq:single-zeta}
  \mathcal Z(2n+1)=f_{2n+1} \in \mathcal U_{\K}.
\end{equation}
For every $q\geq0$, the function $\mathcal Z(\two{q})\in\mathcal U_{\K}$ is constant on $F$.
\end{lemma}

\begin{proof}
Every proper odd cut of $\mathcal Z(2n+1)=\mathcal I(0;1,0^{2n};1)$ as in \cref{prop:cut-formula} has equal endpoints or
a constant interior word, so it vanishes by
\eqref{eq:I-equal-endpoints} and \eqref{eq:I-constant-word}.  Thus
$\partial_m\mathcal Z(2n+1)=0$ for odd $m<2n+1$.  Its counit is zero by
\eqref{eq:odd-constant-term}. By the kernel lemma (\cref{lem:kernel}) it is hence a scalar multiple
of $f_{2n+1}$.  The full-cut coefficient is
\[
I_{\sigma_{2n+1}}(0;1,0^{2n};1)
 =[x^{2n}y]\sigma_{2n+1}=1,
\]
so \eqref{eq:single-zeta} follows.

The extended sequence $(0;(10)^q;1)$ is alternating.  Every odd cut has
equal endpoints, so every $\partial_m$ annihilates $\mathcal Z(\two{q})$.  Every
nonconstant word in $\mathcal U_\K$ has a first letter detected by some
$\partial_m$, hence $\mathcal Z(\two{q})$ is constant.
\end{proof}

\section{Universal level-one coefficients}\label{sec:coefficients}

Choose any element $\sigma_{2k+1}\in\grt_1(\K)_{2k+1}$ of homogeneous weight $2k+1$ such that $[x^{2k}y]\sigma_{2k+1}=1$, as in Definition~\ref{def:normalized-family}. 
Then it turns out that certain coefficients of $\sigma_{2k+1}$ are universal, i.e., they are independent of the choice of $\sigma_{2k+1}$ with these properties. 

\begin{proposition}\label{prop:universal-coefficients}
Let \(k\geq1\), and let
\(
  \sigma_{2k+1}\in\grt_1(\K)_{2k+1}
\)
satisfy \([x^{2k}y]\sigma_{2k+1}=1\).

For \(a,b\geq0\) with \(a+b=k-1\), put
\begin{equation}\label{eq:cab-def}
  c_{a,b}
  :=I_{\sigma_{2k+1}}\bigl(0;\varrho(\two{a}3\two{b});1\bigr)
  =[(xy)^b x^2y(xy)^a]\sigma_{2k+1}.
\end{equation}
Then
\begin{equation}\label{eq:cab-formula}
  c_{a,b}=2(-1)^k
  \left[
    \binom{2k}{2a+2}
    -\left(1-2^{-2k}\right)\binom{2k}{2b+1}
  \right].
\end{equation}
Moreover, the boundary coefficient
\begin{equation}\label{eq:boundary-coefficient}
  c_k^{\partial}
  :=I_{\sigma_{2k+1}}\bigl(0;0,(10)^k;1\bigr)
  =[(xy)^kx]\sigma_{2k+1}
\end{equation}
satisfies
\begin{equation}\label{eq:boundary-value}
  c_k^{\partial}=2(-1)^k.
\end{equation}
\end{proposition}

The proof uses associators, although the statement of the proposition does not.
The required associator identity is recalled next.

\subsection{Terasoma's universal Brown--Zagier identity and proof of Proposition~\ref{prop:universal-coefficients}}

We use the notation \(Z_\Phi(n_1,\ldots,n_r)\in\K\) from
\eqref{eq:Z-Psi-def}.  Thus its argument is a finite sequence of positive
integers.  In particular, \(Z_\Phi(2r+1)\) has the one-term argument
\((2r+1)\), whereas \(\two{q}\) means a sequence of \(q\) copies of \(2\).
We use the convention \(\binom{m}{j}=0\) for \(j>m\).

\begin{theorem}[Terasoma]\label{thm:terasoma}
For every associator \(\Phi\in\Ass(\K)\) and all \(a,b\geq0\),
\begin{align}
  Z_\Phi(\two{a},3,\two{b})
  ={}&2\sum_{r=1}^{a+b+1}(-1)^r
  \left[
    \binom{2r}{2a+2}
    -\left(1-2^{-2r}\right)\binom{2r}{2b+1}
  \right] \notag\\
  &\hspace{25mm}\cdot
  Z_\Phi(2r+1)Z_\Phi(\two{a+b+1-r}).
  \label{eq:terasoma}
\end{align}
\end{theorem}

\begin{proof}[Reference]
This is the Brown--Zagier relation for associators
\cite[Theorem~1.1]{Terasoma}.  Terasoma proves the formula universally for
associators, with rational coefficients.  It therefore remains valid over
every field of characteristic zero.
\end{proof}

\begin{proof}[Proof of \cref{prop:universal-coefficients}]
Extend the prescribed element $\sigma_{2k+1}$ to a normalized
family by choosing, for every odd $m\neq 2k+1$, a normalized
element $\sigma_m\in\mathfrak{grt}_1(\K)_m$.
Construct $\rho_F$ and the coefficient functions using this
family. (The argument below will still depend only on the prescribed element $\sigma_{2k+1}$, so the auxiliary
choices in the other weights play no role.)

Fix \(a,b\geq0\) such that \(k=a+b+1\).
Pull \eqref{eq:terasoma} back along the orbit map
\[
  F\longrightarrow\Ass(\K),
  \qquad h\longmapsto\rho_F(h)\Ihara\Phi_{\mathrm{ev}}.
\]
This gives an identity in the coordinate algebra \(\mathcal U_\K\).

We compare the coefficient of the one-letter word \(f_{2k+1}\).  On the left,
the cut formula \eqref{eq:contiguous-cut}, with a cut of the full length
\(2k+1\), gives
\[
  [f_{2k+1}]\mathcal Z(\two{a},3,\two{b})
  =I_{\sigma_{2k+1}}\bigl(0;\varrho(\two{a}3\two{b});1\bigr)
  =c_{a,b}.
\]
On the right, \cref{lem:single-zeta} gives
\(\mathcal Z(2r+1)=f_{2r+1}\), while every \(\mathcal Z(\two{q})\) is constant on \(F\).
Therefore no summand with \(r<k\) contains \(f_{2k+1}\).  The summand with
\(r=k\) has \(\mathcal Z(\two{0})=1\), and its \(f_{2k+1}\)-coefficient is
\[
  2(-1)^k
  \left[
    \binom{2k}{2a+2}
    -\left(1-2^{-2k}\right)\binom{2k}{2b+1}
  \right].
\]
This proves \eqref{eq:cab-formula}.

It remains to compute \(c_k^{\partial}\).  By \cref{lem:no-pure-words} we have
\[
  \mathcal I^F(0;0;1)=0.
\]
Take the shuffle product with \(\mathcal I^F(0;(10)^k;1)\).  Inserting the additional zero at the beginning gives
\(\mathcal I^F(0;0,(10)^k;1)\).  Every other resulting word has one doubled zero,
and each such word occurs twice.  Hence
\begin{equation}\label{eq:boundary-shuffle}
  0=\mathcal I^F(0;0,(10)^k;1)
    +2\sum_{a+b=k-1}\mathcal I^F(0;(10)^a100(10)^b;1).
\end{equation}
Taking the coefficient of \(f_{2k+1}\) gives
\begin{equation}\label{eq:boundary-sum}
  c_k^{\partial}=-2\sum_{a+b=k-1}c_{a,b}.
\end{equation}
Using \eqref{eq:cab-formula} and the elementary binomial sums
\[
  \sum_{j=1}^{k}\binom{2k}{2j}=2^{2k-1}-1,
  \qquad
  \sum_{j=0}^{k-1}\binom{2k}{2j+1}=2^{2k-1},
\]
one obtains
\(
  \sum_{a+b=k-1}c_{a,b}=(-1)^{k+1}
\).
Equation \eqref{eq:boundary-sum} now gives
\(c_k^{\partial}=2(-1)^k\).
\end{proof}

\subsection{The \texorpdfstring{$2$}{2}-adic size of the coefficients}

For a nonzero rational number \(q\), let the valuation \(v_2(q)\) be the exponent of \(2\)
in its prime factorization, and put \(v_2(0)=+\infty\).
For every odd integer $m=2k+1\geq3$, define its \emph{baseline valuation} by
\begin{equation}\label{eq:baseline-valuation}
  \nu(m)
  :=v_2(c_{0,k-1})
  =2-2k+v_2(k).
\end{equation}
Notice that $\nu(m)\leq0$ for every odd $m\geq3$.

\begin{proposition}\label{prop:2adic}
Let $k\geq1$, let $a,b\geq0$ satisfy $a+b=k-1$, and put
$m=2k+1$. Then:
\begin{enumerate}[label=(\roman*)]
\item $c_{a,b}-c_{b,a}\in2\mathbb Z$ so that $v_2(c_{a,b}-c_{b,a})\geq1$;
\item
$\nu(m)\leq v_2(c_{a,b})\leq0$;
\item $v_2(c_k^{\partial})=1$.
\end{enumerate}
Moreover,
\begin{equation}\label{eq:c-valuation}
  v_2(c_{a,b})
  =
  \nu(m)+v_2\binom{2k-1}{2b}.
\end{equation}
In particular, equality in the lower bound in {\rm(ii)} holds for
$b=0$ and $b=k-1$.
\end{proposition}

\begin{proof}
We compute
\[
  c_{a,b}-c_{b,a}
  =2(-1)^k
  \left(\binom{2k}{2a+2}-\binom{2k}{2b+2}\right),
\]
which proves assertion (i).

Rewrite \eqref{eq:cab-formula} as an even integer plus
\[
  (-1)^k2^{1-2k}\binom{2k}{2b+1}.
\]
The latter number has \(2\)-adic valuation at most zero, so adding an even
integer does not change its valuation.  Thus
\begin{align*}
  v_2(c_{a,b})
  &=1-2k+v_2\binom{2k}{2b+1}\\
  &=2-2k+v_2(k)+v_2\binom{2k-1}{2b},
\end{align*}
which is \eqref{eq:c-valuation}.  The last binomial coefficient is an integer,
so the lower bound in (ii) follows.  The first line gives the upper bound because
\(
  0<\binom{2k}{2b+1}<2^{2k}
\).
For \(b=k-1\) and for $b=0$, the binomial coefficient
\(
  \binom{2k-1}{2k-2}=2k-1 \quad \text{and}\quad \binom{2k-1}{0}=1
\)
are odd, proving \(v_2(c_{0,k-1})=v_2(c_{k-1,0})=\nu(2k+1)\).  Finally,
\eqref{eq:boundary-value} gives (iii).
\end{proof}


\section{Level and deconcatenation}\label{sec:level}

Let \(w\) be a word in the two letters \(2\) and \(3\).  Its weight and level
are
\[
  \wt(w)=2\,\#\{\text{letters }2\text{ in }w\}
         +3\,\#\{\text{letters }3\text{ in }w\},
  \qquad
  \lev(w)=\#\{\text{letters }3\text{ in }w\}.
\]
Under the binary encoding \(\varrho\), every letter \(3\) contributes one
occurrence of \(00\), and no other occurrence of \(00\) appears.  Thus
\(\lev(w)\) is also the number of occurrences of \(00\) in \(\varrho(w)\).
For \(p\geq0\), let
\begin{equation}\label{eq:level-filtration}
  \UZ_{\K,\leq p}
  :=\operatorname{span}_{\K}\{\mathcal Z(w):\lev(w)\leq p\}
  \subset\mathcal U_\K,
\end{equation}
and put \(\UZ_{\K,\leq p}=0\) for \(p<0\).

\begin{lemma}[A cut lowers the level]\label{lem:level-drop}
Every nonzero summand in the cut formula \eqref{eq:contiguous-cut} for \(\partial_m \mathcal Z(w)\) is a scalar
multiple of \(\mathcal Z(u)\) for another word \(u\) in \(2\) and \(3\), and
\[
  \lev(u)\leq\lev(w)-1.
\]
Consequently
\[
  \partial_m\UZ_{\K,\leq p}\subseteq\UZ_{\K,\leq p-1}
\]
for every odd \(m\geq3\).
\end{lemma}

\begin{proof}
Consider the extended binary word
\[
  0\,\varrho(w)\,1.
\]
A nonzero cut has different endpoints, since a cut with equal endpoints has
coefficient zero by \eqref{eq:I-equal-endpoints}.  After the interior of the
cut is deleted, these two endpoints become adjacent.  The new adjacent pair
is therefore \(01\) or \(10\).  No block \(11\), and no run of three zeros,
is created.  The remaining interior word consequently decomposes uniquely
into the blocks \(10\) and \(100\), so it is \(\varrho(u)\) for a word \(u\)
in \(2\) and \(3\).

If the extended cut block contained no occurrence of \(00\), it would be
alternating.  It has an odd number of interior symbols, and hence an odd total
number of symbols including its endpoints.  Its two endpoints would then be
equal, contrary to nonvanishing.  Thus a nonzero cut removes at least one
occurrence of \(00\).  Since the two endpoints are distinct, it creates
no new occurrence of \(00\).
\end{proof}

For a level-one word
\(
  v=\two{a}3\two{b}
\)
put
\begin{equation}\label{eq:cv}
  c(v):=c_{a,b}.
\end{equation}

\begin{proposition}[The top-level part of a cut]\label{prop:deconcat}
Let \(w\) have level \(p\geq1\), and let \(r\geq1\).  Modulo
\(\UZ_{\K,\leq p-2}\),
\begin{equation}\label{eq:deconcat}
  \partial_{2r+1} \mathcal Z(w)
  =\sum_{\substack{w=uv\\
                   \lev(v)=1,\ \wt(v)=2r+1}}
       c(v) \mathcal Z(u)+E_{r,w}.
\end{equation}
Here \(E_{r,w}\in\UZ_{\K,\leq p-1}\) can be written as a linear
combination of level-\((p-1)\) functions \(\mathcal Z(u)\) whose scalar coefficients
are integral linear combinations of
\begin{equation}\label{eq:error-generators}
  c_{a,b}-c_{b,a}\quad(a+b=r-1),
  \qquad c_r^{\partial}.
\end{equation}
In particular, every scalar coefficient occurring in \(E_{r,w}\) has
\(2\)-adic valuation at least one.
\end{proposition}

\begin{proof}
Write
\[
  (a_0,a_1,\ldots,a_N,a_{N+1})=(0,\varrho(w),1).
\]
A summand of the cut formula for \(\partial_{2r+1}\) is determined by an
extended block
\begin{equation}\label{eq:cut-block}
  B_q=(a_q;a_{q+1},\ldots,a_{q+2r+1};a_{q+2r+2}).
\end{equation}
Its scalar coefficient is \(I_{\sigma_{2r+1}}(B_q)\in\K\), and deleting its
interior gives the remaining coefficient function in \(\mathcal U_\K\).
We classify these blocks by the number of occurrences of \(00\), counting
also the pairs that contain an endpoint of the block.

If there is no occurrence of \(00\), the block is alternating.  Its endpoints
are equal, so its coefficient is zero.  If there are at least two occurrences
of \(00\), the deletion lowers the level by at least two.  The corresponding
term belongs to \(\UZ_{\K,\leq p-2}\) and can be discarded.

It remains to consider blocks with exactly one occurrence of \(00\) and with
different endpoints.  Because the full binary word contains no \(11\), the
possibilities with endpoints \(0,1\) are exactly
\begin{equation}\label{eq:one-double-zero-shapes}
  (10)^a100(10)^b\quad(a+b=r-1),
  \qquad
  0(10)^r,
\end{equation}
for the interior of the block.  The possibilities with endpoints \(1,0\)
are their reversals.

The first word in \eqref{eq:one-double-zero-shapes} has coefficient
\(c_{a,b}\).  Unless the block ends at the final endpoint of the full word,
it has a neighboring cut obtained by shifting the block one place to the
right:
\[
0[(10)^a100(10)^b]10 = 01[(01)^a001(01)^b]0.
\]
The two cuts leave exactly the same remaining binary word.  The
shifted cut has coefficient \(-c_{b,a}\): its endpoints are reversed, the
interior word is the reversal of the first shape with \(a\) and \(b\)
interchanged, and its length is odd.  The two contributions therefore combine
to
\[
  (c_{a,b}-c_{b,a})\mathcal Z(u)
\]
for the same remaining word \(u\).  Conversely, every cut of the reflected
first shape occurs in exactly one such pair.
The only first-shape cut without a neighboring partner is the cut ending at
the final endpoint.  Its interior is \(\varrho(v)\) for a level-one suffix
\(v\) of \(w\).  If \(w=uv\), this cut contributes \(c(v)\mathcal Z(u)\).  

The second word in \eqref{eq:one-double-zero-shapes}, and its reversal, has
coefficient \(\pm c_r^{\partial}\).  These terms also belong to the error
term \(E_{r,w}\). Summing all
surviving cuts gives \eqref{eq:deconcat}.  The final valuation statement
follows from \cref{prop:2adic}.
\end{proof}

Thus, on the highest remaining level, \(\partial_{2r+1}\) acts like removal
of a level-one suffix of weight \(2r+1\).  All other contributions have
strictly larger \(2\)-adic valuation.  This is the form of Brown's
``deconcatenation modulo \(I\)'' argument that will be used below
\cite[Section~6]{BrownMTZ}.

\section{Lyndon words and the pairing matrix}\label{sec:lyndon}

\subsection{The two alphabets}

Let
\[
  X=\{3,5,7,\ldots\},\qquad 3<5<7<\cdots,
\]
and identify the letter \(m\in X\) with the source generator \(s_m\).
Let \(Y=\{3,2\}\), ordered by \(3<2\) (sic!).  For an alphabet \(A\), let \(A^*\) denote the set of finite words in
\(A\), including the empty word.  Words are ordered lexicographically, with
a proper prefix smaller than the longer word.  The weight of a word is the
sum of its letters.  Define a map on words
\begin{equation}\label{eq:phi}
  \phi:X^*\longrightarrow Y^*,
  \qquad
  \phi(2n+1)=3\,2^{\{n-1\}}.
\end{equation}
The map preserves weight, and the length of a word in \(X\) equals the level
of its image.

A nonempty word is a \emph{Lyndon word} if it is strictly smaller than each
of its nonempty proper suffixes.  Let \(\operatorname{Lyn}_{N,p}(X)\) be the
set of Lyndon words in \(X\) of weight \(N\) and length \(p\).  Define
\(\operatorname{Lyn}_{N,p}(Y)\) similarly, using level \(p\) in place of
length.

\begin{lemma}\label{lem:lyndon-encoding}
For \(p\geq1\), the map \(\phi\) is order preserving and restricts to a
bijection
\[
  \phi:\operatorname{Lyn}_{N,p}(X)
  \xrightarrow{\sim}\operatorname{Lyn}_{N,p}(Y).
\]
\end{lemma}

\begin{proof}
When two words in \(X\) first differ at letters \(2m+1<2n+1\), the code
\(3\,2^{\{m-1\}}\) ends before \(3\,2^{\{n-1\}}\).  At that point the shorter code
is either finished or followed by a \(3\), while the longer code continues
with a \(2\).  Since \(3<2\), the order is preserved.

A suffix of \(\phi(u)\) that starts at the beginning of a code block is the
image of a suffix of \(u\).  A suffix that starts inside a code block begins
with \(2\), whereas \(\phi(u)\) begins with \(3\), and is therefore larger.
This proves that \(u\) is Lyndon if and only if \(\phi(u)\) is Lyndon.
A Lyndon word in \(Y\) with positive level must begin with \(3\); otherwise
a suffix beginning with \(3\) would be smaller.  Its inverse image is
therefore obtained by cutting immediately before each letter \(3\).
\end{proof}

For a Lyndon word \(u\) in \(X\), let \(\operatorname{br}(u)\) be its
standard Lyndon bracketing.  The Lyndon basis theorem states that the
\(\operatorname{br}(u)\) form a basis of the free Lie algebra on \(X\), and
that their expansions in the free associative algebra have the form
\begin{equation}\label{eq:lyndon-triangular}
  \operatorname{br}(u)
  =u+\sum_{v>u}n_{u,v}v,
  \qquad n_{u,v}\in\Z.
\end{equation}
Every word \(v\) in the sum is a permutation of the letters of \(u\), see
\cite{Reutenauer}.  We extend \(v\mapsto\partial_v\) (as in \eqref{eq:partial-word}) linearly from words to
the free associative algebra.  In particular,
\(\partial_{\operatorname{br}(u)}\) is obtained by replacing each letter
\(m\) by \(\partial_m\) and each Lie bracket by the commutator of operators.

\subsection{Triangularity for iterated cuts}
For words \(u,u'\in X^*\) of the same weight and length, define
\begin{equation}\label{eq:raw-pairing}
  M(u',u):=
  \bigl(\partial_u \mathcal Z(\phi(u'))\bigr)(1_F)\in\K.
\end{equation}
We interpret $M(u',u)$ as the entries of a matrix. We will later need to analyze the invertibility of this matrix. To this end, we study the $2$-adic valuation of its entries, following one main idea of Brown \cite{BrownMTZ}.

Extend the baseline valuation $\nu$ of \eqref{eq:baseline-valuation} additively from the alphabet
$X=\{3,5,7,\ldots\}$ to the free monoid $X^*$:
\begin{equation}\label{eq:word-baseline}
  \nu(\varnothing)=0,
  \qquad
  \nu(m_1\cdots m_p)
  :=\sum_{j=1}^p \nu(m_j).
\end{equation}

\begin{proposition}\label{prop:word-pairing}
Let \(u,u'\in X^*\) have the same weight and length.  Then
\(M(u',u)\in\Q\), and
\begin{align}
  v_2(M(u',u))&\geq \nu(u), \label{eq:word-pairing 1}\\
  u'<u&\Longrightarrow v_2(M(u',u))>\nu(u), \label{eq:word-pairing 2}\\
  v_2(M(u,u))&=\nu(u) \label{eq:word-pairing 3}.
\end{align}
For \(u=(2i_1+3)\cdots(2i_p+3)\), the unique contribution of valuation
\(\nu(u)\) to \(M(u,u)\) is
\begin{equation}\label{eq:diagonal-product}
  \prod_{j=1}^p c_{0,i_j}.
\end{equation}
\end{proposition}

\begin{proof}
Let $u=(2i_1+3)\cdots(2i_p+3)$.
We want to apply \cref{prop:deconcat} successively.
The operators in
\(\partial_{2 i_j+3}\) act one after the other on $\mathcal Z(u')$, producing intermediate terms proportional to $\mathcal Z(u'')$ for some intermediate words $u''$.
If $u''$ is one such word obtained after applying the $j$ operators $\partial_{2 i_{p-j+1}+3}\cdots \partial_{2 i_p+3}$, then it has level $\leq p-j$. If it has level $<p-j$, then the application of each further operator will reduce the level by one more, and \cref{lem:level-drop} shows that the term
is eventually killed.
So we have to consider only those terms in the action of each $\partial_{2 i_j+3}$ that reduce the level by exactly $1$, and those terms are those covered by \cref{prop:deconcat}.
Call the first term on the right of \eqref{eq:deconcat} the \emph{principal term}, and the second term \(E_{r,w}\) the \emph{error term}.
The principal term removes a level-one suffix
\(v=\two{a}3\two{b}\); for the letter \(2i+3\) that is acting, one has
\(a+b=i\), and the scalar factor is \(c_{a,b}\).  An \emph{error} term has a
scalar factor which is an integral linear combination of
\(c_{a,b}-c_{b,a}\) and \(c_{i+1}^{\partial}\). 
Since all these coefficients are rational it is clear that \(M(u',u)\in\Q\).
Furthermore, the contribution of the principal term to the valuation of the coefficient is \(v_2(c_{a,b})\geq \nu(2i+3)\), while the contribution of an error term is at least one, which is strictly larger than \(\nu(2i+3)\), see \cref{prop:2adic}. The inequality \eqref{eq:word-pairing 1} follows immediately.
By the same reasoning we also see that equality in \eqref{eq:word-pairing 1} cannot occur if at some step we use the error term.
It hence remains to understand summands containing only principal factors.  Write
\[
  u=(2i_1+3)\cdots(2i_p+3),
  \qquad
  u'=(2i'_1+3)\cdots(2i'_p+3).
\]
The successive suffix cuts partition
\[
  \phi(u')=3\,2^{\{i'_1\}}\,3\,2^{\{i'_2\}}\cdots3\,2^{\{i'_p\}}
\]
into \(p\) consecutive factors, each containing one letter \(3\).  From left
to right they have the form
\[
  3\,2^{\{b_1\}},\quad
  2^{\{a_2\}}3\,2^{\{b_2\}},\quad\ldots,\quad
  2^{\{a_p\}}3\,2^{\{b_p\}},
\]
where
\begin{equation}\label{eq:factor-boundaries}
  a_1=0,
  \qquad b_j+a_{j+1}=i'_j\quad(1\leq j<p),
  \qquad b_p=i'_p,
\end{equation}
and
\begin{equation}\label{eq:factor-weights}
  i_j=a_j+b_j\quad(1\leq j\leq p).
\end{equation}
The first relation gives
\(
  i_1=b_1=i'_1-a_2\leq i'_1
\).
If equality holds, then \(a_2=0\).  Repeating the argument shows that either
at the first place where the two sequences differ one has \(i_j<i'_j\), or
all inequalities are equalities.  Hence an all-principal contribution can
occur only when \(u\leq u'\). Thus if $u>u'$ then we must pick up an error factor so that \eqref{eq:word-pairing 2} follows.

If \(u=u'\), all the boundary shifts \(a_2,\ldots,a_p\) must be zero.  The
partition is then unique, its factors are the individual code blocks
\(3\,2^{\{i_j\}}\), and its scalar coefficient is exactly
\eqref{eq:diagonal-product}.  Each factor \(c_{0,i_j}\) has valuation
\(\nu(2i_j+3)\). The last two assertions follow.
\end{proof}

\begin{lemma}[A valuation criterion for invertibility]\label{lem:matrix-criterion}
Let \(A=(a_{ij})\) be a square matrix over \(\Q\), with its rows and columns
indexed by the same finite totally ordered set.  Suppose that for every
column \(j\) there is an integer \(d_j\) such that
\[
  v_2(a_{ij})\geq d_j\quad\text{for all }i,
  \qquad
  v_2(a_{jj})=d_j,
\]
and that the first inequality is strict whenever \(i<j\).  Then \(A\) is
invertible.
\end{lemma}

\begin{proof}
In the determinant expansion, the diagonal product has valuation
\(\sum_jd_j\).  Every nonidentity permutation uses at least one entry
strictly above the diagonal.  Hence every other product in the determinant expansion has
valuation strictly larger than \(\sum_jd_j\).  The diagonal product cannot
be cancelled, and the determinant is nonzero.
\end{proof}

\subsection{The Lyndon pairing}

Fix a weight \(N\) and a length \(p\geq1\), and order
\(\operatorname{Lyn}_{N,p}(X)\) increasingly.  Define the square matrix
\begin{equation}\label{eq:pairing-matrix}
  P_{N,p}(u',u)
  :=\bigl(\partial_{\operatorname{br}(u)}\mathcal Z(\phi(u'))\bigr)(1_F),
  \qquad
  u,u'\in\operatorname{Lyn}_{N,p}(X).
\end{equation}

\begin{proposition}\label{prop:pairing-invertible}
For every \(N\) and \(p\geq1\), the matrix \(P_{N,p}\) has rational
entries and is invertible over \(\Q\).  It is therefore invertible over every field of
characteristic zero.
\end{proposition}

\begin{proof}
Fix a column indexed by a Lyndon word \(u\).  The leading word in
\eqref{eq:lyndon-triangular} is \(u\).  By \cref{prop:word-pairing}, its
"raw" column (without the sub-leading terms from \eqref{eq:lyndon-triangular}) has the following properties:
\[
  u'\longmapsto M(u',u)
\]
has valuation at least \(\nu(u)\) everywhere, strictly larger than \(\nu(u)\)
in rows \(u'<u\), and exactly \(\nu(u)\) in the diagonal row \(u'=u\).

Every other word \(v\) in \eqref{eq:lyndon-triangular} satisfies \(v>u\) and
is a permutation of the letters of \(u\).  Thus \(\nu(v)=\nu(u)\).  For a row
\(u'<u\), and also for the diagonal row \(u'=u\), one has \(u'<v\).
\Cref{prop:word-pairing} therefore shows that the contribution of
\(\partial_v\) has valuation strictly larger than \(\nu(u)\) in all these
rows.  Consequently the full column, including the sub-leading terms from \eqref{eq:lyndon-triangular}, has the same diagonal leading
term and the same strict inequality above the diagonal as the raw column.

The hypotheses of \cref{lem:matrix-criterion} are satisfied with
\(d_j=\nu(u)\) in the column indexed by \(u\).  Hence \(P_{N,p}\) is
invertible.
\end{proof}

\section{Proof of Theorem \ref{thm:main}}\label{sec:proof}
By rescaling the $\sigma_{2k+1}$, we may assume that they form a normalized family as in Definition \ref{def:normalized-family}.  Let
\[
  \rho_{\mathrm{Lie}}:
  \Lhat_{\K}(s_3,s_5,s_7,\ldots)\longrightarrow\grt_1(\K)
\]
be the map of \eqref{eq:source-lie-map} constructed using this family.
Suppose that $0\neq L\in \Lhat_{\K}(s_3,s_5,s_7,\ldots)$ is in the kernel of \(\rho_{\mathrm{Lie}}\), i.e., $\rho_{\mathrm{Lie}}(L)=0$.
Because \(\rho_{\mathrm{Lie}}\) preserves weight, a nonzero element of its
kernel has a nonzero homogeneous component in the kernel, and hence we may assume that \(L\) is homogeneous of weight \(N\).

Decompose \(L\) according to bracket length, meaning the number of generator
letters in a Lie monomial:
\[
  L=L_p+L_{p+1}+\cdots,
  \qquad L_p\neq0,
\]
where \(p\) is minimal.  By the Lyndon basis theorem,
\begin{equation}\label{eq:Lp}
  L_p=
  \sum_{u\in\operatorname{Lyn}_{N,p}(X)}
  a_u\operatorname{br}(u),
  \qquad a_u\in\K.
\end{equation}
Since
\(\rho_{\mathrm{Lie}}(L)=0\),
\[
  \rho_F(\exp(tL))
  =\exp\bigl(t\rho_{\mathrm{Lie}}(L)\bigr)=1.
\]
Every orbit function is therefore constant along left multiplication by
\(\exp(tL)\), and in particular we have, for any fixed \(u'\in\operatorname{Lyn}_{N,p}(X)\).  
\begin{equation}\label{eq:DL-zero}
  D_L\mathcal Z(\phi(u'))=0.
\end{equation}

A Lie monomial of bracket length \(q\) acts, up to the nonzero sign in
\eqref{eq:anti-sign}, as a sum of compositions of \(q\) operators
\(\partial_m\).  Every nonzero operator lowers the level by at least one, by
\cref{lem:level-drop}.  The function \(\mathcal Z(\phi(u'))\) has level \(p\).
After \(p\) nonzero derivatives, only level-zero functions remain; these are
linear combinations of the functions \(\mathcal Z(\two{a})\), which are constant by
\cref{lem:single-zeta}.  Every further derivative is zero.  It follows that
all components \(L_q\) with \(q>p\) act trivially on \(\mathcal Z(\phi(u'))\).
Equation \eqref{eq:DL-zero} therefore reduces to
\[
  D_{L_p}\mathcal Z(\phi(u'))=0.
\]

Every summand in \eqref{eq:Lp} has bracket length \(p\).  By
\eqref{eq:anti-sign}, evaluation at \(1_F\) gives, up to the common nonzero
sign \((-1)^{p-1}\),
\[
  \sum_{u\in\operatorname{Lyn}_{N,p}(X)}
  P_{N,p}(u',u)a_u=0
\]
for $P_{N,p}(u',u)$ the entries of the pairing matrix of \eqref{eq:pairing-matrix}.
This holds for every row \(u'\).  The matrix \(P_{N,p}\) is invertible by
Proposition~\ref{prop:pairing-invertible}, so every \(a_u\) is zero.  This contradicts
\(L_p\neq0\). We conclude that \(\rho_{\mathrm{Lie}}\) is injective.

\hfill\qed


\begin{thebibliography}{99}

\bibitem{APS}
A. Alekseev, M. Podkopaeva, P. Ševera.
\newblock On rational Drinfeld associators.
\newblock {\it Selecta Math.} (N.S.) 17 (2011), no. 1, 47--65. 

\bibitem{BrownMTZ}
F.~Brown,
\emph{Mixed Tate motives over $\Z$},
Ann.\ of Math. (2) \textbf{175} (2012), 949--976.

\bibitem{BCGP}
F.~Brown, M.~Chan, S.~Galatius, S.~Payne.
\newblock Hopf algebras in the cohomology of $\mathcal A_g$, $\mathrm{GL}_n(\Z)$, and $\mathrm{SL}_n(\Z)$.
\newblock arXiv:2405.11528.

\bibitem{Drinfeld}
V.~G.~Drinfeld,
\emph{On quasitriangular quasi-Hopf algebras and on a group that is closely
connected with $\operatorname{Gal}(\overline{\Q}/\Q)$},
Leningrad Math. J. \textbf{2} (1991), 829--860.

\bibitem{Reutenauer}
C.~Reutenauer,
\emph{Free Lie Algebras},
London Mathematical Society Monographs, New Series 7,
Oxford University Press, 1993.

\bibitem{Terasoma}
T.~Terasoma,
\emph{Brown--Zagier relation for associators},
arXiv:1301.7474, 2013.

\bibitem{Willwacher}
T.~Willwacher,
\emph{The Grothendieck--Teichm\"uller Group},
lecture notes, 2014, especially Chapter~7.

\bibitem{Zagier}
D.~Zagier,
\emph{Evaluation of the multiple zeta values
$\zeta(2,\ldots,2,3,2,\ldots,2)$},
Ann. of Math. (2) \textbf{175} (2012), 977--1000.

\end{thebibliography}
\end{document}